\documentclass[10pt]{article}

\usepackage{amssymb}
\usepackage{amsmath}
\usepackage{amsthm}
\usepackage{enumerate}
\usepackage{verbatim}
\usepackage{color}

\newtheorem{theorem}{Theorem}
\newtheorem{definition}[theorem]{Definition}
\newtheorem{lemma}[theorem]{Lemma}

\newtheorem{corollary}[theorem]{Corollary}

\newcommand{\Var}{\mbox{\rm Var}}
\def\qed{\unskip\nobreak\hfill\penalty50\hskip 3pt\hbox{}\nobreak
\hfill\hbox{\vrule width 4 pt height 10 pt}}

\begin{document}
\bibliographystyle{plain}

\title {Entrance Time and R\'enyi Entropy}
\author{ Chinmaya Gupta \thanks {Mathematics Department, University of Houston.
E-mail: $<$ccgupta@math.uh.edu$>$.} \and Nicolai
Haydn\thanks{Mathematics Department, USC, Los Angeles, 90089-1113.
E-mail: $<$nhaydn@math.usc.edu$>$.} \and Milton Ko \thanks
{Mathematics Department, USC, Los Angeles, 90089-1113. E-mail:
$<$miltonko@gmail.com$>$.} \and Erika A. Rada-Mora \thanks{
Institute of Mathematic and Statistic, University of S\~ao Paulo ,
05508-090,  Brazil. E-mail:$<$alejarada@gmail.com$>$. The fourth
author was supported by CNPq-Brazil Procs. 143256/2009-2 and
140114/2012-2.}}

\maketitle

\begin{abstract} For ergodic systems with generating partitions, the well known result of Ornstein
and Weiss shows that the exponential growth rate of the recurrence time is almost surely equal to
the metric entropy. Here we look at the exponential growth rate of entrance times, and show that it
equals the entropy, where the convergence is in probability in the product measure.
This is however under the assumptions that the limiting entrance times distribution exists almost surely.
This condition looks natural in the light of an example by Shields in which the limsup in the exponential
growth rate is infinite almost everywhere but where the limiting entrance times do not exist.
We then also consider $\phi$-mixing systems and prove a result connecting the R\'enyi entropy
to sums over the entrance times orbit segments.
\end{abstract}

\section{Introduction}
Let $T$ be a map on a space $\Omega$, then $\{T^i(x) \}_{i=0}^{\infty}$ defines the orbit of $x \in \Omega$. For  a set $A \subset \Omega$, the \emph{entrance time} $\tau_A$ of a point $x$ into the set $A$ refers to the time that takes for the orbit of $x$ to first enter the set $A$.  In particular, if $x \in A$, $\tau_A$ refers to the \emph{return time} of the point $x$: the time that it takes the orbit of  $x$  to return for the first time to the set $A$.
For invariant probability measures $\mu$ the Poincar\'e Recurrence Theorem states that a point in a positive
measure set returns to that set almost surely. In other words $\tau_A(x)<\infty$ for almost
every $x$ in $A$, provided $\mu(A)>0$. In 1946 this result was quantified by Kac who
showed that for ergodic measures the expected return time is the reciprocal of the measure of the return
set. If the space $\Omega$ has a generating partition $\cal A$ then in 1993, Ornstein and Weiss~\cite{OW}
proved for ergodic measures $\mu$ that $\frac 1 n \log \tau_n$ converges to the entropy $h_{\mu}$ almost surely,  where the {\em $n$-th recurrence time} $\tau_n(x)=\tau_{A_n(x)}(x)$ measures the time for $x$   to return to its initial $n$-cylinder  $A_n(x)$.

Intuitively, the entrance time should behave similar to the return time in ergodic systems, as in such systems when a point $x$ travels long enough it tends to forget where it started. If we assume $\Omega$ has
 a partition $\cal A$, then it is natural to consider the exponential growth rate of entrance times to
 the $n$-cylinders $A_n(z)$ centred at an arbitrary point $z$.
However, Shields~\cite{PS} in 1992 constructed an example  in which $\frac1n\log \tau_{A_n(z)}(x)$ does
 not converge for almost every $x$. In fact the $\limsup$ goes to infinity almost surely.
 Here we impose an additional assumption in order to get convergence in probability to the metric
 entropy. We require that the limiting entrance times distributions exist almost everywhere.

We then also give a condition under which the convergence of is almost surely. We then
also look at $\phi$-mixing measures and show that they satisfy this conditions and thus
have almost sure convergence of exponential growth rate of entrance times.
In the last theorem we consider the  \emph{R\'enyi entropy} which was first introduced by
Alfr\'ed R\'enyi~\cite{RE} in 1961 in order to generalize the Shannon entropy. Here we
generalize a result of  Ko~\cite{Ko} which had been proven for return times to entrance times.
For $\phi$-mixing systems we obtain in Theorem~\ref{main3} a relationship between entrance time
and the R\'enyi entropy.

In section~\ref{results}, we state definitions, basic facts and the four main theorems that we will
prove in this paper. Theorem~\ref{main1} proves the convergence of the entrance time in
probability while Theorem~\ref{main2} proves the almost sure convergence of entrance time under
an additional assumption. Theorem~\ref{bridge} verifies that additional assumption for $\phi$-mixing
measures. Theorem ~\ref{main3} considers the sum of measures of $n$-cylinders
 visited by a point along its orbit until it enters a set, and proves that it converges to a constant in
 terms of the R\'enyi entropy and metric entropy for $\phi$-mixing systems. The proofs of
 Theorem~\ref{main1} and~\ref{main2} are given in Section~\ref{proof.entrance.time},
 the proof of Theorem~\ref{bridge} is in section~\ref{proof.bridge} while the proof of Theorem~\ref{main3} is given in Section~\ref{EMR}.

\section{Main Results} \label{results}
Let $\Omega$ be a space with a probability measure $\mu$ and $T:\Omega \rightarrow \Omega$
be a measurable map. We assume $\mu$ is $T$-invariant and ergodic. Let ${\cal A} = \{ {\cal P}_i \}$
 be a generating partition (finite or countably infinite) and denote by
  ${\cal A}^n = \bigvee_{i=0}^{n-1}T^{-i}{\cal A}= \left \{ \bigcap_{0 \leq i \leq n-1}T^{-i}({\cal P}_{j_i}) : {\cal P}_{j_i} \in {\cal A} \right \}$ its $n$-th join. The elements of $\mathcal{A}^n$ are referred to as $n$-cylinders.
 We denote by $A_n(x) \in {\cal A}^n$ the $n$-cylinder which contains the point $x\in\Omega$.

The theorem of Shannon-McMillan-Breiman (see e.g.~\cite{Man}) states that for any $T$-invariant ergodic probability measure $\mu$ and generating partition $\cal A$ of $\Omega$,
\begin{equation} \label{SMB}
\lim_{n \rightarrow \infty} \frac{1}{n} \left | \log \mu(A_n(x)) \right | = h_{\mu}
\end{equation}
for almost every $x \in \Omega$, where $h_\mu$ is the measure theoretic entropy of $\mu$.
 This asymptotic formula was first proven by Shannon~\cite{Shannon} in 1948 for
stationary Markov chains and then subsequently strengthened by McMillan and Breiman to its present
form for finite alphabets and then extended to countably infinite alphabets (with finite entropy) by
Chung~\cite{Chung} in 1961 and Carleson~\cite{Carleston}  in 1958. In other words, the measure of the $n$-cylinder which contains $x$ decays exponentially with rate roughly the metric entropy.

For any $x \in \Omega$ and set $A \subset \Omega$, let us define now the {\em entrance time}
of $x$  into the set $A$ by
$$
\tau_A(x) = \min \{ i \geq 1 : T^i(x) \in A \}.
$$
We call $\tau_n(x) = \tau_{A_n(x)}(x)$ the  $n$-th {\em recurrence time} of $x$; it is
the first time that $x$ returns to the $n$-cylinder which contain $x$. Ornstein and Weiss proved  in~\cite{OW}
for finite partition, and in \cite{OW2} for countably infinite partition (provided that $h_{\mu}$ is finite) that for almost every $x$,
\begin{equation} \label{OW1}
\lim_{n \rightarrow \infty} \frac{1}{n} \log \tau_n(x) = h_{\mu}
\end{equation}
assuming $\mu$ is ergodic.
Intuitively, the entrance time $\tau_{A_n(z)}(x)$ should behave similarly to the recurrence time
$\tau_n(z) = \tau_{A_n(z)}(z)$ as in~(\ref{OW1}), since when points travel a long enough time in
ergodic systems they tend to forget where they start and hence whether starting at the point $x$ or $z$
should not matter. However, Shields constructed in 1992 an example of a dynamical system in
which the entrance time fails to converge~\cite{PS}. Here we prove that $\frac1n \log \tau_{A_n(z)}(x)$
converges in probability to $h_\mu$ provided the system has an almost sure entrance times distribution.

In the following we adopt probability notations that for events $A, B\subset \Omega$ we denote $\mu(A )$ by $\mathbb{P}(A)$ and $\mu_B(A) = \mu(B\cap  A) / \mu(B)$ by $\mathbb{P}_B(A)$
(assuming $\mu(B)>0$). For $z \in \Omega$, $n \in \mathbb{N}$ and $t > 0$, put
$$
F_z^n(t) = \mathbb{P}\left( \tau_{A_n(z)} \geq \frac{t}{\mu(A_n(z))} \right) = \mu \left( \left \{x \in \Omega : \tau_{A_n(z)}(x) \geq \frac{t}{\mu(A_n(z))} \right \} \right)
$$
and if $B = A_n(z)$ we put
$$
F_B(t) = F_z^n(t). \label{def:F_B}
$$
We shall require that the limit $\lim_{n\rightarrow\infty}F_z^n$ exists almost everywhere.
For a number of classes of positive entropy systems this limit is $e^{-t}$ a.s.. There are
however examples of ergodic zero entropy systems that have other limiting distributions.

The following two theorems prove convergences of the entrance time: Theorem~\ref{main1} proves
the existence of the limit and convergence in probability under the assumption that the limiting
entrance (or return) times exists almost everywhere. Theorem~\ref{main2} gives a sufficient condition
under which the convergence is almost sure. Let us note that there are many examples when
the limiting entrance/return times do not exist. The example of Shields is one of them. Also,
Downarowicz~\cite{Down} has given examples when the limiting distribution exists along subsequences
of full density and where the limit can to be made to decay arbitrarily slowly, in particular so
slow as to violate the condition in Theorem~\ref{main2}.

\begin{theorem} \label{main1}
Suppose for almost every $z \in \Omega$ and for $t \geq 0$, $\lim_{n \rightarrow \infty} F_z^n(t) = F_z(t)$ exists and $F_z(t)\to 0$ as $t\to\infty$. 
Then $\frac{1}{n} \log \tau_{A_n(z)}(x)$ converges to $h_{\mu}$ in probability as $n$ goes to $\infty$.
\end{theorem}

\begin{theorem} \label{main2}
Suppose $\mu$ is a $T$-invariant ergodic probability measure on $\Omega$, and for all small enough $\epsilon > 0$ we have
$$
\sum_{n=1}^{\infty} \int_{\Omega} F_z^n(e^{n\epsilon}) \, d \mu(z) < \infty.
$$
Then
$$
\lim_{n \rightarrow \infty} \frac{1}{n}\log \tau_{A_n(z)}(x) = h_{\mu}
$$
for $\mu \times \mu$-almost every $(x,z) \in \Omega\times \Omega$.
\end{theorem}

\noindent{\bf Remark.}  (i) Let us note that summability condition of Theorem~\ref{main2} is only required
 to get the upper bound on the limit. By Lemma~\ref{entimelower} we get the lower bound on the
 limit almost surely for all ergodic measures. \\
 (ii) Although the recurrence time $\tau_n(x) = \tau_{A_n(x)}(x)$ is a special case of the return time,
  Theorem \ref{main2} does not imply the asymptotic formula in~(\ref{OW1}) since the above convergence is
  true for $\mu\times\mu$-almost every $(x,z)$ which that does not imply that it applies to points on the
  diagonal $x=z$ as the  diagonal has measure $0$ in the product measure.

The remainder of the paper looks at a situation in which the hypothesis of Theorem~\ref{main2} is
satisfied. We consider systems with some mixing property.

\begin{definition} We say an invariant measure $\mu$ is {\em $\phi$-mixing} if
there exists a decreasing function $\phi : \mathbb{N} \rightarrow \mathbb{R}$ so that
\begin{equation} \label{mixing}
\frac{| \mu (A \cap T^{-(n+i)}(B) ) - \mu (A) \mu (B)|}{\mu (A)} \leq \phi (i)
\end{equation}
for all $A \in {\cal A}^n$, all $B \in \sigma({\cal A}^*)$, where ${\cal A}^*=\cup_{n=1}^{\infty} {\cal A}^n$  and for all $n \in \mathbb{N}$.
\end{definition}

\noindent In the following two theorems will moreover assume that $\phi$ is summable, that is
 $\sum_{i=1}^{\infty} \phi(i) < \infty$. Let us note that the limiting entrance times distribution $F_z(t)$
 for $\phi$-mixing measures (with summable $\phi$) is exponential almost
 everywhere~\cite{Abadi04}, i.e.\ $F_z(t)=e^{-t}$ for $\mu$-almost
 every $z\in\Omega$. This includes in particular measures of maximal entropy and equilibrium states
 for H\"older continuous potential on Axiom~A systems which are $\psi$-mixing at an exponential
 rate.

\begin{theorem} \label{bridge}
Suppose $\mu$ is a $T$-invariant $\phi$-mixing measure of $\Omega$ with summable $\phi$.
 Then
$$
\lim_{n \rightarrow \infty} \frac{1}{n}\log \tau_{A_n(z)}(x) = h_{\mu}
$$
for $\mu \times \mu$-almost every $(x,z) \in \Omega  \times \Omega$.
\end{theorem}

\noindent For the final result we will also require that the (countably infinite) partition
 ${\cal A} = \{{\cal P}_i\}_{i=1}^{\infty}$ has an {\em exponentially decaying tail} if
\begin{equation}\label{exponential_tail}
\mu \left(\bigcup_{i \geq j} {\cal P}_i \right) = { \cal O} (\delta ^j)
\end{equation}
for all $j$ and for some $\delta < 1$. If $|{\cal A}|$ is finite then (\ref{exponential_tail}) is trivially satisfied.

\noindent  For $s>0$, put
$$
Z_n(s) = \sum_{A_n \in {\cal A}^n} \mu (A_n)^{1+s}
$$
and define the {\em R\'enyi Entropy Function}~\cite{RE} on $(0,\infty)$ by
$$
R(s) = \lim_{n \rightarrow \infty} \frac{1}{sn}| \log Z_n(s)|
$$
if the limit exists. For larger values of $s$, the
R\'enyi entropy is weighted towards  highest probability events. Moreover for the value $s=0$,
the R\'enyi entropy typically coincides with the Shannon entropy. The R\'enyi entropy exists as a uniform limit in weakly $\psi$-mixing systems \cite{HV} and a pointwise limit under weaker assumption \cite{Ko}.

\begin{theorem}\label{main3}
Suppose $T:\Omega \to \Omega$ is measurable, $\mu$ is $T$ invariant and $\phi$-mixing with summable $\phi$, and $\mathcal{A}$ has exponential tails. Suppose the R\'enyi entropy $R(s)$ exists for $s>0$. Then for $\mu\times \mu$ every $(x, z)\in \Omega\times \Omega$,
\[
\lim_{n\to\infty} \frac{1}{n} \log \sum_{i = 1}^{\tau_{A_n (z)}(x)} \mu(A_n(T^i(x)))^s = h_\mu - s R(s).
\]

\end{theorem}

This generalises a previous result of Ko~\cite{Ko} in which $z$ was assumed to be equal to $x$.
Obviously, (\ref{mixing}) ensures the ergodicity of $\mu$. Furthermore, (\ref{mixing}) implies
the exponential decay of cylinders and this ensures that the metric entropy $h_{\mu}$ is positive.
The summability of $\phi$ is needed to estimate the variance of the hitting time function (see Section \ref{hitnum}).
The condition (\ref{exponential_tail}) in particular implies that $h_{\mu}$ is finite (See Lemma 4 of \cite{Ko}).
It also allows us to control the ``tail" of the partition ${\cal A}^n$
 in the proof of Lemma \ref{part2}. From now on we will abbreviate  $\tau_{A_n(z)}(x)$ by $\tau_n^z(x) $ for convenience.


\section{Convergence of Entrance Time}\label{proof.entrance.time}

We first prove the lower bound of Theorem \ref{main1} and \ref{main2}.
\begin{lemma} \label{entimelower}
Suppose $\mu$ is a $T$-invariant ergodic probability measure of $\Omega$. Then
$$
\liminf_{n \rightarrow \infty} \frac{1}{n}\log \tau_n^z(x)  \geq h_{\mu}
$$
for $\mu \times \mu$-almost every $(x,z) \in \Omega \times \Omega$.
\end{lemma}

\begin{proof}
 Let $0< b < c < h_{\mu}$, and put
$$
\mathbb{E}_n = \{x : \tau_n^z(x)   \leq e^{bn} \}.
$$
Note that $\mathbb{E}_n = \bigcup_{j = 1}^{[e^{bn}]} T^{-j} (A_n(z))$. Then we have,
$$
\mu(\mathbb{E}_n) \leq \sum_{j=1}^{[e^{bn}]} \mu(T^{-j}(A_n(z))) = \mu(A_n(z)) e^{bn}.
$$
By (\ref{SMB}), $\mu(A_n(z)) \leq e^{{\color{red} -}nc}$ for almost every $z$. Therefore, $\mu(\mathbb{E}_n) \leq e^{-(c-b)n}$, summable on $n$. By the Borel Cantelli Lemma, for almost every $z$, $\mu( \limsup \mathbb{E}_n) = 0$. In words, this implies that for almost every $z$, the set of initial conditions $x$ for which the return times to $A_n(z)$ are smaller than $e^{bn}$ infinitely often have $\mu$ measure 0.
This implies
$$
\liminf_{n \rightarrow \infty} \frac{1}{n}\log \tau_n^z(x)   \geq h_{\mu}
$$
for almost every $x$.
\end{proof}

\noindent{\bf Remark.} Note that in the proof above we showed that for any $\varepsilon > 0$ and almost every $z$,
$$
\lim_{n \rightarrow \infty} \mu \left( \left \{x : \frac{1}{n} \log \tau_n^z(x)   \leq h_{\mu}- \varepsilon \right \} \right) = 0,
$$
which is equivalent to
$$
\lim_{n \rightarrow \infty} \mu \times \mu \left( \left \{(x,z) : \frac{1}{n} \log \tau_n^z(x)   \leq h_{\mu}- \varepsilon \right \} \right) = 0.
$$
To complete the proof of Theorem \ref{main1} and \ref{main2}, we obtain the other side of the inequality in Lemma \ref{entimelower} under certain assumptions. One might have attempted to show this by only assuming that the measure $\mu$ is $T$-invariant and ergodic. However, Shields \cite{PS} constructed an example of a dynamical system (on a four-element subshift) in which
$$
\limsup_{n \rightarrow \infty}  \frac{1}{n} \log \tau_n^z(x)   = \infty
$$
for $\mu \times \mu$-almost every $(x,z) \in \Omega \times \Omega$.

\begin{lemma} \label{entimeupper4}
Suppose for all small enough $\epsilon > 0$ and $\delta>0$,
$$
\lim_{n \rightarrow \infty} \mu( \left \{ x \in \Omega: F_z^n(e^{n\epsilon}) > \delta \right \}) = 0
$$
for almost every $z \in \Omega$. Then $\frac{1}{n} \log \tau_n^z(x)  $ converges to $h_{\mu}$ in probability as $n \to \infty$.
\end{lemma}

\begin{proof}  Let $\delta > 0$, $b > h_{\mu}$ and $\mathbb{D}_n = \{(x,z) \in \Omega \times \Omega : \tau_n^z(x)   > e^{nb} \}$. We want to show that $\mu \times \mu (\mathbb{D}_n)$ is bounded from above by $\delta$ for large enough $n$. Further let $\epsilon \in (0,  b - h_{\mu})$ and $\delta' = b -(h_{\mu} + \epsilon)$. Put
$$
\Omega_n = \{ z \in \Omega : F_z^n(e^{n \epsilon}) \leq \delta /3 \} ,
$$
$$
{\cal S}_n = \{ B \in {\cal A}_n: \mu(B) \geq e^{-n(h_{\mu} + \delta')} \}
$$
and
$$
S_n = \bigcup_{B \in {\cal S}_n} B.
$$
By hypothesis and~(\ref{SMB}), we can choose $n$ large enough so that $\mu(\Omega_n^c) < \delta /3$ and $\mu(S_n^c) < \delta/3$. Put
$$
\bar{\Omega}_n = \{ B \in {\cal A}^n : B \cap \Omega_n \neq \emptyset \}.
$$
As $F_z^n(t)$ is locally constant on $n$-cylinders $B$, $F_B(e^{n \epsilon}) \leq \delta /3$ for $B \in \bar{\Omega}_n$. Note also that $F_B(t)$ decreases as $t$ increases; therefore,  $F_B(\mu(B) e^{nb}) \leq F_B(e^{n\epsilon})$. It follows that for large enough $n$
\begin{eqnarray*}
\mu \times \mu(\mathbb{D}_n) &=& \sum_{B \in {\cal A}^n} \mu(B) \mathbb{P}(\tau_{B} \geq e^{nb})\\
& = & \sum_{B \in {\cal A}^n} \mu(B) F_B(\mu(B) e^{nb})\\
& = & \sum_{B \in {\bar{\Omega}_n^c}} \mu(B) F_B(\mu(B) e^{nb}) + \sum_{B \in \bar{\Omega}_n} \mu(B) F_B(\mu(B) e^{nb}) \\
& < & \mu(\Omega_n^c) + \sum_{B \in {\cal S}_n^c} \mu(B) F_B(\mu(B) e^{nb}) + \sum_{B \in \bar{\Omega}_n \cap {\cal S}_n} \mu(B) F_B(\mu(B) e^{nb}) \\
& < & \frac{\delta}{3} + \mu(S_n^c) + \sum_{B \in \bar{\Omega}_n \cap {\cal S}_n} \mu(B) F_B(e^{n\epsilon}) \\
& < & \frac{2 \delta}{3} + \frac{\delta}{3} \sum_{B \in \bar{\Omega}_n \cap {\cal S}_n} \mu(B) \leq \delta
\end{eqnarray*}
As the above is true for any $b > h_{\mu}$, we showed for any $\varepsilon > 0$
$$
\lim_{n \rightarrow \infty} \mu \times \mu \left( \left \{(x,z) : \frac{1}{n} \log \tau_n^z(x)   \geq h_{\mu} + \varepsilon \right \} \right) = 0.
$$
Together with the remark under Lemma \ref{entimelower}, the proof is completed.

\end{proof}

Similar to the entrance time distribution, for $z \in \Omega$, $n \in \mathbb{N}$
 and $t > 0$ we define the {\em return time distribution} as
$$
\tilde{F}_z^n(t) = \mathbb{P}_{A_n(z)}\left( \tau_{A_n(z)} \geq \frac{t}{\mu(A_n(z))} \right) = \mu \left( \left \{x \in A_n(z) : \tau_n^z(x)   \geq \frac{t}{\mu(A_n(z))} \right \} \right) / \mu(A_n(z))
$$
assuming $\mu(A_n(z)) > 0$ and if $B = A_n(z)$ we put
$$
\tilde{F}_B(t) =\tilde{F}_z^n(t).
$$
By~\cite{HLV} the entrance times distribution $F_B$ and the return times distribution $\tilde{F}_B$
 are related by the identity $F_B(t)=\int_t^\infty\tilde{F}_B(s)\,ds$.
 
 \vspace{3mm}

\noindent{\bf Proof of Theorem \ref{main1}.} Let $\beta$, $\epsilon$ and $\delta$ be positive. By Lemma \ref{entimeupper4}, we want to show that
$$
\lim_{n \rightarrow \infty} \mu( \left \{ x \in \Omega: F_z^n(e^{n\epsilon}) > \delta \right \}) = 0.
$$
Put $V_N = \{z \in \Omega: F_z(N) \leq \delta /2 \}$. Since $F_z(t)$ decreases to 0 by assumption, 
there exists $K=N_{\delta, \beta}$ such that $\mu( V_K^c) < \beta /2$. Put
$U_n = \{z \in \Omega: | F_z^n(K) - F_z(K) |  \leq \delta /2 \}$.  Since $F_z^n$ converges to $F_z$ for almost every $z$, when $n$ is large enough, we have $\mu(U_n^c) < \beta /2$, and $e^{n \epsilon} > K$. For $z \in V_K \cap U_n$, we get
$$
F_z^n(e^{n \epsilon}) \leq F_z^n(K) \leq F_z(K) + \delta /2 < \delta.
$$
This shows for large $n$,
$$
\mu(z \in \Omega: F_z^n(e^{n \epsilon}) > \delta) \leq \mu(V_K) + \mu(U_n) < \beta/2 + \beta/2 = \beta,
$$
and the proof is completed. \qed

\vspace{3mm}

\noindent Now we turn to prove the almost sure convergence of the entrance time.

\vspace{3mm}

\noindent{\bf Proof of Theorem \ref{main2}.} Let $b > h_{\mu}$, $\epsilon \in (0,  b - h_{\mu})$ and $\delta = b -(h_{\mu} + \epsilon)$. We claim that
$$
\limsup_{n \rightarrow \infty} \frac{1}{n}\log \tau_n^z(x)   \leq h_{\mu}
$$
for $\mu \times \mu$-almost every $(x,z) \in \Omega \times \Omega$. Put
$$
{\cal S}_n = \{ B \in {\cal A}_n: \mu(B) \geq e^{-n(h_{\mu} + \delta)} \}
$$
and
$$
S_n = \bigcup_{B \in {\cal S}_n} B.
$$
Then as $F_B(t)$ is decreasing, if we put $\mathbb{D}_n = \{(x,z) \in \Omega \times \Omega : \tau_n^z(x)   > e^{nb} \}$, we have
\begin{eqnarray*}
\mu \times \mu(\mathbb{D}_n \cap (\Omega \times S_n)) & = & \sum_{B \in {\cal S}_n} \mu(B) F_B(\mu(B)e^{nb}) \\
& \leq &  \sum_{B \in {\cal S}_n} \mu(B) F_B(e^{-n(h_{\mu} + \delta)} e^{nb}) \\
& = & \sum_{B \in {\cal S}_n} \mu(B) F_B(e^{n\epsilon}) = \int_{\Omega} F_z^n(e^{n\epsilon}) \, d \mu(z)
\end{eqnarray*}
which is summable by our hypothesis. Applying the Borel Cantelli Lemma and (\ref{SMB}) gives
$$
\mathbb{P}\left(\limsup_{n \rightarrow \infty} \mathbb{D}_n \right) \leq \mathbb{P}\left(\limsup_{n \rightarrow \infty}  (\mathbb{D}_n \cap (\Omega \times S_n))\right) + \mathbb{P}\left(\limsup_{n \rightarrow \infty} (\Omega \times S_n^c) \right) = 0.
$$
As $b>h_\mu$ is arbitrary our claim is proved. Together with Lemma~\ref{entimelower}, we proved
Theorem~\ref{main2}. \qed

\begin{corollary} \label{entimeupper2}
Suppose for almost every $z \in \Omega$, there exists $F_z(t)$, a decreasing function on $t>0$, and a summable sequence $a_n > 0$ such that for all small enough $\epsilon > 0$,
\begin{itemize}
\item[(i)] $\sum_{n=1}^{\infty} \mu(\{ z: |F_z^n (e^{n \epsilon}) - F_z(e^{n \epsilon})| > a_n \}) < \infty$
and
\item[(ii)] $\sum_{n=1}^{\infty} \int_{\Omega} F_z(e^{n\epsilon}) \, d \mu(z) < \infty.$
\end{itemize}
Then
$$
\lim_{n \rightarrow \infty} \frac{1}{n}\log \tau_n^z(x)   = h_{\mu}
$$
for $\mu \times \mu$-almost every $(x,z) \in \Omega \times \Omega$.
\end{corollary}
\noindent{\bf Proof.} In light of Theorem \ref{main2}, it is sufficient to show that our hypothesis implies
$$
\sum_{n=1}^{\infty} \int_{\Omega} F_z^n(e^{n\epsilon}) \, d \mu(z) < \infty
$$
for small enough $\epsilon$. But
\begin{eqnarray*}
& & \int_{\Omega} F_z^n(e^{n\epsilon}) \, d \mu(z)\\
& \leq & \int_{\Omega} |F_z^n(e^{n\epsilon}) - F_z(e^{n\epsilon})| d \mu(z) + \int_{\Omega} F_z(e^{n\epsilon}) \, d \mu(z) \\
& \leq & 2 \mu(\{ z: |F_z^n (e^{n \epsilon}) - F_z(e^{n \epsilon})| > a_n \}) + a_n + \int_{\Omega} F_z(e^{n\epsilon})\,  d \mu(z).
\end{eqnarray*}
The three terms on the right hand side above are all summable by our hypothesis, and we are done.
\qed

\section{Proof of Theorem~\ref{bridge}}\label{proof.bridge}

We shall need the following result of Abadi~\cite[Theorem 1]{Abadi04}). The following is a simplified
version.

\begin{lemma} \label{abadi}
Let $\mu$ be a $\phi$-mixing $T$-invariant probability measure such that
$\phi$ is summable. Then there exist a constants $M>0, K_{\ref{abadi}}<\infty$ such that

$$
\mathbb{P}\left ( \tau_A > \frac{t}{\mu(A)} \right) \leq e^{-Mt} + K_{\ref{abadi}}(n\mu(A) + \phi(n))
$$
for all $A\in\mathcal{A}^n$ and all $n\in\mathbb{N}$.
\end{lemma}

\noindent {\bf Proof of Theorem~\ref{bridge}.}
We have to prove that the limit
$$
\lim_{n \rightarrow \infty} \frac{1}{n}\log \tau_n^z(x)   = h_{\mu}
$$
exists for $\mu \times \mu$-almost every $(x,z) \in \Omega \times \Omega$ under the assumption
that $\mu$ is $\phi$-mixing and $\sum_i\phi(i)<\infty$.
By Theorem \ref{main2}, we need to show that
$$
\sum_{n=1}^{\infty} \int_{\Omega} F_z^n(e^{n\epsilon}) \, d \mu(z) < \infty
$$
for any small enough $\epsilon$. It is well known that for a $\phi$-mixing  system, there exists $r > 0$ such that $\mu(A) \leq e^{-rn}$ for all $n$ and $n$-cylinder $A \in {\cal A}^n$. Moreover by Lemma~\ref{abadi}
we have
$$
\mathbb{P}\left ( \tau_{A_n(z)} > \frac{t}{\mu(A_n(z))} \right) \leq e^{-Mt} + K_{\ref{abadi}}(n\mu(A_n(z)) + \phi(n))
$$
for every $z\in\Omega$, $n\in\mathbb{N}$ and $t > 0$.
Then for any $\epsilon > 0$,
\begin{eqnarray*}
\sum_{n=1}^{\infty} \int_{\Omega} F_z^n(e^{n\epsilon}) \, d \mu(z) & = & \int_{\Omega} \sum_{n=1}^{\infty} \mathbb{P} \left( \tau_{A_n(z)} > \frac{e^{n\epsilon}}{\mu(A_n(z))} \right) \, d \mu(z) \\
& \leq & \int_{\Omega} \left \{ e^{-M\exp(n \epsilon)} + K_{\ref{abadi}}\mu(A_n(z))
+ K_{\ref{abadi}} \phi(n) \right \} \, d \mu(z) \\
& \leq & \int_{\Omega} \left \{ e^{-M\exp(n \epsilon)} + K_{\ref{abadi}}e^{-rn}
+ K_{\ref{abadi}} \phi(n) \right \} \, d \mu(z) < \infty \\
\end{eqnarray*}
as required. \qed

\section{Proof of Theorem~\ref{main3}} \label{EMR}

From now on we will assume that the measure $\mu$ satisfies the $\phi$-mixing property with summable $\phi$ and the partition $\cal A$ has an exponentially decaying tail (see (\ref{mixing}) and (\ref{exponential_tail})).
We separately prove the upper and lower bound on the limit. The upper is quite easy but the lower
bound requires a more careful analysis of hitting numbers.
For $z \in \Omega$, define
$$
D^z:= \left \{x: \lim_{n \rightarrow \infty} \frac{1}{n} \log \tau_n^z(x)   = h_{\mu} \right \}
$$
and put
$$
D :=  \{z: \mu(D^z) = 1 \}.
$$
Theorem~\ref{bridge} implies that $\mu(D) = 1$. For $z \in D$,  $\epsilon > 0$ and all $x \in D^z$, we have
\begin{equation}\label{OW}
e^{n(h_{\mu} - \epsilon)} < \tau_n^z(x)   < e^{n(h_{\mu}+ \epsilon)}
\end{equation}
for large enough $n$. In the rest of the paper we assume $z \in D$, and for convenience we put $W_n^s(x,z) = \sum_{i=1}^{\tau_n^z(x)  }\mu (A_n(T^i(x)))^s$. Note that $W_n^0 (x,x) = \tau_n(x)$, and in this case Theorem \ref{main3} coincides with (\ref{OW1}). Also the case $x=z$ and $s>0$ of Theorem \ref{main3} was proven in \cite{Ko}.

\subsection{\bf Proof of the upper bound of the limit in Theorem~\ref{main3}}
By the proof of Proposition 2.3 in~\cite{DSUZ} (see also~\cite{Ko} Proposition 6) for every $\epsilon > 0$ there exists
 $D_{\epsilon} \subset \Omega$ with measure 1 such that for $x \in D_{\epsilon}$,
$$
\limsup_{n \rightarrow \infty} \frac{1}{n} \log \sum_{i=1}^{\exp(n(h_{\mu}+ \frac{\epsilon}{3}))}\mu (A_n(T^i(x)))^s \leq h_{\mu} - sR(s) + \epsilon.
$$
Also by (\ref{OW}), we know that for $z \in D$ and therefore for all $x \in D^z\cap D_\epsilon$,
$$
\limsup_{n \rightarrow \infty} \frac{1}{n} \log W_n^s(x,z) \leq \limsup_{n \rightarrow \infty} \frac{1}{n} \log \sum_{i=1}^{\exp(n(h_{\mu}+ \frac{\epsilon}{3}))}\mu (A_n(T^i(x)))^s
\leq h_{\mu} - sR(s) + \epsilon .
$$
Finally, as $\mu \left ( \bigcap_{m=1}^{\infty} D_{1/m} \cap D^z \right ) = 1$, this establishes the upper
bound in Theorem~\ref{main3}.
 \qed

\subsection{Hitting numbers}\label{hitnum}
To prove the lower bound on the limit in Theorem~\ref{main3} we need estimates on the \emph{hitting number}
$$
N_{U,M}(x) = \sum_{i=0}^M \chi_U \circ T^{i}(x)
$$
of $U \in \sigma({\cal A}^n)$ (unions of $n$-cylinders),  where $\chi_U$ is the characteristic function of
the set $U$. $N_{U,M}(x)$ counts the number of times  $i\in[0,M]$ that $T^i(x) \in U$.
Similarly $\nu_x^z(U) =N_{U,\tau_n^z(x)  }(x)$ is the number of times that $x$ hits the set $U$ when it travels along its orbit segment until it returns to $A_n(z)$.
Following~\cite{DSUZ} it was shown in~\cite{Ko} that the variance of the hitting time can be estimated by
 $\Var(N_{U,M}) \leq c_1Mn \mu(U)$ for a constant $c_1$.

The following two lemmas provide us with lower and upper bounds for the
hitting time. For $z=x$ these results have been proven in~\cite{Ko} and here we give the modification
required for the present more general setting.

\begin{lemma}\label{lowerbound}
Let $\mu$ be an $\phi$-mixing $T$-invariant measure where $\phi(i)$ is summable and $U_n \in \sigma({\cal A}^n)$, $n = 1,2,...$, be a sequence of sets in $\sigma({\cal A}^n)$. Let $\epsilon > 0$ and assume $\gamma_n$ is a sequence of positive numbers so that for all $n$ large enough ($C,a,b>0$ constants):
Assume one of the following two conditions are satisfied:\\
(I) $\mu (U_n) \geq C e^{-\gamma _n (h_{\mu} + \epsilon)}$ and
$(n-\gamma _n)h_{\mu} - \epsilon (n+ \gamma_n) \geq an^b$,\\
(II) $\mu (U_n) \geq C e^{-\gamma_n (h_{\mu} - \epsilon)}$ and
$(n-\gamma_n)(h_{\mu} - \epsilon) \geq an^b$.

Then for almost every $(x,z)$,
$$
N_{U_n, \tau_n^z(x)  }(x) \geq \frac{\mu (U_n)}{2} e^{n(h_{\mu} - \epsilon)}
$$
for all $n$ large enough.
\end{lemma}

\noindent{\bf Proof}. (I) Put $M =  [e^{n(h_{\mu} - \epsilon)}]$. Using the estimate on the variance of
$N_{U_n,M}$ and Chebycheff's inequality it was shown in~\cite{Ko}  that there exists a set  $D'$, with measure 1,  such that, for all  $x \in D'$, and for all  $n$ large enough, it holds,
$\frac{N_{U_n,M}(x)}{M} \geq \mu (U_n)/2$.  Since by (\ref{OW}) for $z \in D$,
$x \in D^z \cap D'$, and for $n$ large enough, we have $\tau_n^z(x)   > e^{n(h_{\mu} - \epsilon)}$,
and therefore
$$
N_{U_n, \tau_n^z(x)  }(x) \geq N_{U_n,M}(x) \geq M \mu(U_n) /2 =  \mu (U_n) e^{n(h_{\mu} - \epsilon)}/2.
$$
Since $\mu(D^z \cap D') = 1$, the estimate follows.

Part (II) is proven similarly.
 \qed


\begin{lemma}\label{upperbound}
Let $\mu$ be as in Lemma \ref{lowerbound} and $U_n \in \sigma({\cal A}^n)$, $n = 1,2,...$, be a sequence of sets. Suppose there exists a constant $C >0$ so that $\mu (U_n) \geq C$ for all large enough $n$. Then for $\epsilon > 0$ and for almost every $x$,
\begin{displaymath}
N_{U_n, \tau_n^z(x)  }(x) \leq \frac{3 \mu (U_n)}{2} e^{n(h_{\mu} + \epsilon)}
\end{displaymath}
for all $n$ large enough.
\end{lemma}

\noindent{\bf Proof}. For $M = [e^{n(h_{\mu} + \epsilon)}]$ it was shown in~\cite{Ko} that
$\left | \frac{N_{U_n,M}(x)}{M} - \mu (U_n) \right| \leq \mu (U_n)/2$
for all $n$ large enough. By (\ref{OW}) for $z \in D$, $x \in D^z \cap D'$, and for large enough $n$, we have $\tau_n^z(x)   < e^{n(h_{\mu} + \epsilon)}$, and hence $N_{U_n, \tau_n^z(x)  }(x) \leq N_{U_n,M}(x) \leq 3M \mu(U_n) /2 =  3 \mu (U_n) e^{n(h_{\mu} + \epsilon)}/2$ as desired. \qed

Using (\ref{SMB}) and Egoroff's Theorem, there exists a set $\cal E$ with measure greater than $1/2$ on which
$ |\log \mu(A_n(x))| /n$ converges to $h_{\mu}$ uniformly as $n\rightarrow\infty$. Define
$$
E_n := \{ x: A_n(x) \cap {\cal E} \neq \emptyset \},
$$
the union of those $n$-cylinders which intersect $\cal E$. As $E_n \in \sigma({\cal A}^n)$, let us apply
Lemma~\ref{lowerbound} and~\ref{upperbound} to obtain estimations on the hitting number of $E_n$.

\begin{corollary}\label{En}
For any positive $\epsilon < h_{\mu}$ and almost every $(x,z)$,
\begin{itemize}
\item[(I)] $\nu_x^z(E_n) \geq \frac{\mu(E_n)}{2} e^{n(h_{\mu} - \epsilon)}$;
\item[(II)] $\nu_x^z(E_n^c) \leq \frac{3 \mu(E_n^c)}{2} e^{n(h_{\mu} + \epsilon)}$
\end{itemize}
for all $n$ large enough (where  $\nu_x^z(U) = N_{U,\tau_n^z(x)  }(x)$).
\end{corollary}

\noindent{\bf Proof}. (I) We use Lemma \ref{lowerbound}(II) with
$U_n = E_n$ and $\gamma_n = n/2$ for any $a,b > 0$. In order to
verify the second part of the condition of
Lemma~\ref{lowerbound}(II) note that $\mu(E_n) \geq \mu({\cal E})
\geq 1/2$ for all $n$. Thus $\mu (E_n) \geq  e^{-\gamma_n (h_{\mu} -
\epsilon)}$ for all large enough $n$ and Lemma~\ref{lowerbound}(II)
 gives the desired result.
 (II) First suppose that $\mu(E_m) < 1$ for some integer $m$. Since $E_{n+1} \subseteq E_n$
 (as $A_{n+1}(x) \subseteq A_n(x)$) we conclude that $\mu(E_n^c) \geq \mu(E_m^c)$ for all $n>m$.
 Hence by Lemma \ref{upperbound} (with $C = \mu(E_m^c)$) we are done.
 If $\mu(E_n) = 1$ for all $n$, then put
$$
B=\bigcap_{n \geq 1} \bigcap_{i \geq 0} T^{-i}(E_n).
$$
Then $\mu(B) = 1$ and for $x \in B$, $T^i(x) \notin E_n^c$ for any $n$. Hence for almost every $x$, $\nu_x^z(E_n^c) = 0 = \frac{3 \mu(E_n^c)}{2} e^{n(h_{\mu} + \epsilon)}$ for all $n$. \qed


\subsection{Proof of the lower bound of the limit in Theorem~\ref{main3}}
Let $c, \alpha \in (0,1)$, and $\epsilon>0$ be a small number which depends on $h_{\mu}$ and $c$ (which is close to 1), and will be determined later. Define $\tilde{\gamma}_n =n - [n^{\alpha}]$ and $\Delta = [n^{\alpha}]$.
Denote by $\tilde{A}_n(x) \in {\cal A}^{\tilde{\gamma}_n}$ and $\bar{A}_n(x) \in {\cal A}^{[cn]}$ the
 $\tilde{\gamma}_n$-cylinder and the $[cn]$-cylinder which contain $x$ respectively.
  As $n > \tilde{\gamma}_n > cn$ (for $n$ large enough), we have
   $A_n(x) \subset \bar{A}_n(x) \subset \tilde{A}_n(x)$. In the following we denote an $n$-cylinder by
   $A_n$ or $A$, a $[cn]$-cylinder by $\bar{A}_n$ or $\bar{A}$, and a $\tilde{\gamma}_n$-cylinder by $\tilde{A}_n$ or $\tilde{A}$. For $\epsilon > 0$, there exists $K_{\epsilon}$ such that for any $n > K_{\epsilon}$ , we have
\begin{equation} \label{uniformality}
e^{-n(h_{\mu} + \epsilon)} \leq \mu(A_n(x)) \leq e^{-n(h_{\mu} - \epsilon)}
\end{equation}
for $x \in {\cal E}$. Note that for $n>K_{\epsilon}$ and if the $n$-cylinder $A \subset E_n$, we have $A=A_n(x)$ for some $x \in {\cal E}$ and hence $\mu(A)$ satisfies (\ref{uniformality}). From now on, we assume $n$ is large enough so that $n > \tilde{\gamma}_n > cn > K_{\epsilon}$ and hence (\ref{uniformality}) holds with $n$ replaced by $\tilde{\gamma}_n$ and $cn$. The inequality (\ref{uniformality}) shows the uniformity property of the measures of cylinders in the sense that when $x \in {\cal E}$, we have
\begin{equation}\label{regular}
\mu (\bar{A}_n(x))  \leq e^{-cn(h_{\mu}-\epsilon)}
 \leq  \mu (A_n(x))e^{(1-c)nh_{\mu}+ 2n \epsilon}.
\end{equation}
If we put
$$
\bar{E}_n  =  \{x: \bar{A}_n(x) \cap {\cal E} \neq \emptyset \},
$$
(the union of $[cn]$-cylinders which intersect $\cal {\cal E}$) then  $E_n \subseteq \bar{E}_n \subseteq {\cal E}$.

Let $\varepsilon = \frac{1-c}{1+c} h_{\mu}$ (recall that $h_{\mu}$ is positive, by the comment after (\ref{exponential_tail})), and from now on we choose $\epsilon < \varepsilon$.
When we let $\varepsilon \rightarrow 0$, we have both $\epsilon \rightarrow 0$ and $c \rightarrow 1$.
For convenience we also put
\begin{eqnarray*}
\tilde{E}_n^+ & = & \{\tilde{A} \in {\cal A}^{\tilde{\gamma}_n} : \mu( \tilde{A}) \geq e^{-\tilde{\gamma}_n(h_{\mu} - \epsilon)} \}; \\
\tilde{E}_n^- & = & \{\tilde{A} \in {\cal A}^{\tilde{\gamma}_n} : \mu( \tilde{A}) \leq e^{-\tilde{\gamma}_n(h_{\mu} - \epsilon)} \}.
\end{eqnarray*}
According to~\cite{Ko} the lower bound on the limit in Theorem~\ref{main3} follows immediately from the following two lemmas.

\begin{lemma} \label{part1}
There exists some constant $K_{\ref{part1}}$, which depends only on $s$, so that for almost every $x$,
$$
W_n^s(x,z) \geq e^{K_{\ref{part1}}n\varepsilon} e^{nh_{\mu}} \sum_{\tilde{A} \in \tilde{E}_n^-} \mu (\tilde{A})^{1+s}
$$
for all $n$ large enough.
\end{lemma}

\noindent{\bf Proof.} We proceed in three steps:\\
(I) We have
$$
W_n^s(x,z) \geq  \sum_{i=1}^{\tau_n^z(x)  } \mu (A_n(T^i(x)))^s \chi_{E_n}(T^i(x))
 \geq  e^{s((c-1)nh_{\mu} - 2n \epsilon)} S_1
 \geq  e^{-4sn\varepsilon} S_1
$$
as  $\epsilon,1-c<\varepsilon$, where
$S_1 = \sum_{i=1}^{\tau_n^z(x)  } \mu (\bar{A}_n(T^i(x)))^s \chi_{E_n}(T^i(x))$.\\
(II) Put $S_2 =  \sum_{i=1}^{\tau_n^z(x)  } \mu
(\bar{A}_n(T^i(x)))^s \chi_{\bar{E}_n}(T^i(x))$. Then $S_1 \leq S_2$
as $E_n \subseteq \bar{E}_n$. If $T^i(x) \in E_n$, then by
(\ref{uniformality}) we have $\mu(\bar{A}_n(T^i(x))) \geq
e^{-cn(h_{\mu} + \epsilon)}$ and consequently by Corollary
\ref{En}(I) (as $\mu(\bar{A}_n(T^i(x)))\ge e^{-n(h_\mu+\epsilon)}$
for all $x\in E_n$) 
\begin{equation}\label{S1}
S_1  \geq  \sum_{i=1}^{\tau_n^z(x)  } e^{-scn(h_{\mu} + \epsilon)} \chi_{E_n}(T^i(x))
 =  \nu_x^z(E_n) e^{-scn(h_{\mu} + \epsilon)}
 \geq  \frac{\mu(E_n)}{2}e^{n(h_{\mu} - \epsilon)} e^{-scn(h_{\mu} + \epsilon)}.
\end{equation}
Meanwhile, by Corollary \ref{En}(II) (as $\mu(\bar{A}_n(T^i(x)))\le e^{-n(h_\mu-\epsilon)}$ for all $x\in E_n$)
we also have
\begin{eqnarray}\label{S2}
S_2 - S_1 & = & \sum_{i=1}^{\tau_n^z(x)  } \mu (\bar{A}_n(T^i(x)))^s \chi_{\bar{E}_n \backslash E_n}(T^i(x))  \nonumber \\
& \leq & \sum_{i=1}^{\tau_n^z(x)  } e^{-scn(h_{\mu} - \epsilon)} \chi_{{E_n}^c}(T^i(x)) \nonumber \\
& = & \nu_x^z(E_n^c) e^{-scn(h_{\mu} - \epsilon)} \nonumber \\
& \leq & \frac{3\mu(E_n^c)}{2}e^{n(h_{\mu} + \epsilon)} e^{-scn(h_{\mu} - \epsilon)}.
\end{eqnarray}
Since $\mu({E_n}^c) < 1/2 < \mu(E_n)$, we get
$$
S_2 -S_1 \leq 3 \left [ \frac{1}{2} \mu(E_n)e^{n(h_{\mu} - \epsilon)} e^{-scn(h_{\mu} + \epsilon)} \right ]
e^{2n \epsilon}  e^{2scn \epsilon}.
$$
By (\ref{S1}) the quantity in the bracket of the above inequality is less than $S_1$. Consequently
 $S_2 - S_1 \leq 3 e^{2n \epsilon(1+sc)}S_1$ and
$$
S_2  \leq   e^{4n \epsilon(1+sc)}S_1
 \leq  e^{4(1+s) \varepsilon n}S_1
$$
as $c < 1$ and $\epsilon < \varepsilon$. \\
(III) We also have
$$
S_2 = \sum_{\bar{A} \subset \bar{E}_n} \nu_x^z(\bar{A}) \mu(\bar{A})^s
$$
using the counting function $\nu_x^z(\bar{A})$ for which we have bounds by Lemma \ref{lowerbound}(I):
since for $\bar{A} \subset \bar{E}_n$, one has $\mu (\bar{A}) \geq e^{-cn(h_{\mu} + \epsilon)}$, and hence the first condition of Lemma \ref{lowerbound}(I) is fulfilled (with $\gamma_n = cn$). The second condition,
 $(n-cn)h_{\mu} - \epsilon(n+cn)$ $=$ $((1-c)h_{\mu} - \epsilon (1+c))n > 0$ follows from
 $\epsilon < \varepsilon = \frac{(1-c)}{(1+c)}h_{\mu}$. Therefore, by Lemma \ref{lowerbound}(I),
$$
S_2  \geq  \sum_{\bar{A} \subset \bar{E}_n} \frac{\mu (\bar{A})}{2} e^{n(h_{\mu} - \epsilon)} \mu(\bar{A})^s
 =  \frac{1}{2} e^{n(h_{\mu} - \epsilon)} \sum_{\bar{A} \subset \bar{E}_n}\mu(\bar{A})^{1+s}
$$
and we conclude as in~\cite{Ko} that
$S_2 \geq  e^{n(h_{\mu}-\varepsilon)} S_3$, where
$S_3 = \sum_{\tilde{A} \subset \bar{E}_n} \mu(\tilde{A})^{1+s}$.
Finally we use the fact from~\cite{Ko} that  $S_3 \geq e^{-2sn\varepsilon} S_4$, where $S_4 = \sum_{\tilde{A} \in \tilde{E}_n^-} \mu (\tilde{A})^{1+s}.$

We thus obtain for some constant $c_1$ independent of $\varepsilon$:
$$
W_n^s(x,z)\ge e^{-c_1n\varepsilon} e^{nh_{\mu}} S_4.
$$
 \qed

\begin{lemma} \label{part2}
There exists some constant $K_{\ref{part2}}$, which depends only on $s$, so that for almost every $x$,
$$
W_n^s(x,z) \geq e^{-K_{\ref{part2}}n\varepsilon}e^{nh_{\mu}} \sum_{\tilde{A} \in \tilde{E}_n^+} \mu (\tilde{A})^{1+s}
$$
for all $n$ large enough.
\end{lemma}


\noindent{\bf Proof.} Let $\beta > 1$ and define ($\Delta = [n^{\alpha}]$)
\begin{displaymath}
{\cal G}_n=\left \{x: \mu (A_n(x)) \geq \exp({-\Delta^{\beta}}) \mu (\tilde{A}_n(x)) \right \}.
\end{displaymath}
Then ${\cal G}_n$ is a union of $n$-cylinders, since by definition if $x \in {\cal G}_n$, we have $A_n(x) \subseteq {\cal G}_n$. Moreover put
$$
F_{j, \Delta} = \bigcap_{i = 1}^{\Delta} T^{-i}\left ( \bigcup_{m=1}^{j-1} {\cal P}_m \right )
$$
for $j \in \mathbb{N}$. Note that if $x \in F_{j, \Delta}$, then for $1 \leq i \leq \Delta$, $T^i(x) \notin {\cal P}_k$ for all $k \geq j$. The set $F_{j, \Delta}$ is a finite union of $\Delta$-cylinders and consists of point $x$ that do not hit the ``tail" $\bigcup_{m = j}^{\infty} {\cal P}_m$ for the first $\Delta$ iterates. Obviously, $F_{j, \Delta} \subseteq F_{j+1, \Delta}$ and $F_{j, \Delta'} \subseteq F_{j, \Delta}$ for $\Delta' > \Delta$. We will consider the sets $F_{k_n, \Delta}$ for $k_n = [n^t], t > 1$.

We make use of ${\cal G}_n$ to compare the summands $\mu(A_n(T^i(x)))$ and $\mu(\tilde{A}_n(T^i(x)))$ as follows:
\begin{eqnarray} \label{claim1}
W_n^s(x,z)& \geq & \sum_{i=1}^{\tau_n^z(x)  } \mu (A_n(T^i(x)))^s  \chi_{{\cal G}_n}(T^i(x)) \nonumber \\
& \geq &e^{-s \Delta ^{\beta}} \sum_{i=1}^{\tau_n^z(x)  } \mu (\tilde{A}_n(T^i(x)))^s \chi_{{\cal G}_n}(T^i(x)) \nonumber \\
& = & e^{-s \Delta ^{\beta}} \sum_{\tilde{A} \in {\cal A}^{\tilde{\gamma}_n}} \mu (\tilde{A})^s \nu_x^z(\tilde{A} \cap {\cal G}_n) \nonumber \\
& \geq &e^{-s \Delta ^{\beta}} \sum_{ \tilde{A} \in \tilde{E}_n^+} \mu (\tilde{A})^s \nu_x^z(\tilde{A} \cap {\cal G}_n).
\end{eqnarray}
The first inequality is true since $0 \leq \chi_{{\cal G}_n} \leq 1$. The second inequality follows from the definition of ${\cal G}_n$. The last inequality is valid since we restrict the sum to a subcollection of $\tilde{A}$.
In order to apply Lemma \ref{lowerbound}(II) to obtain a lower bound of $\nu_x^z(\tilde{A} \cap {\cal G}_n)$, we write
\begin{eqnarray}
\mu(\tilde{A} \cap {\cal G}_n) \nonumber & \geq & \mu \left (\tilde{A} \cap T^{-\tilde{\gamma}_n} (F_{k_n, \Delta}) \cap {{\cal G}_n} \right) \nonumber \\
& = & \left [ \frac{\mu \left( \tilde{A} \cap T^{-\tilde{\gamma}_n}(F_{k_n, \Delta}) \right)}{\mu(\tilde{A})} - \frac{\mu \left( \tilde{A} \cap T^{-\tilde{\gamma}_n}(F_{k_n, \Delta}) \cap {\cal G}_n^c \right)}{\mu(\tilde{A})} \right ] \mu(\tilde{A}). \label{good}
\end{eqnarray}
In~\cite{Ko} it was shown that the quantity inside the bracket goes to 1 as $n$ tends to $\infty$
(The first term converges to $1$ and the second term converges to $0$.)
Thus $\frac{\mu(\tilde{A} \cap {\cal G}_n)}{\mu(\tilde{A})}\rightarrow1$ as $n\rightarrow\infty$
and in particular for large enough $n$,
$\mu(\tilde{A} \cap {\cal G}_n) \geq \mu(\tilde{A}) /2 \geq e^{-\tilde{\gamma}_n(h_{\mu} - \epsilon)}/2$
if $\tilde{A} \in \tilde{E}_n^+$. Let us now apply Lemma \ref{lowerbound}(II) where we put $\gamma_n = \tilde{\gamma}_n$.  Hence
$$
\nu_x^z(\tilde{A} \cap {\cal G}_n) \geq \frac{\mu(\tilde{A} \cap {\cal G}_n)}{2} e^{n(h_{\mu} - \epsilon)}
 \geq \frac{\mu( \tilde{A})}{4} e^{n(h_{\mu} - \epsilon)}.
$$
From (\ref{claim1}) one thus obtains
$$
W_n^s(x,z) \geq \frac{\exp(-s\Delta^{\beta})}{4} \exp(-n \epsilon) \exp(nh_{\mu})\sum_{\tilde{A} \in \tilde{E}^+_n} \mu(\tilde{A})^{1+s}.
$$
Now let $\beta \in (1,  1/ \alpha)$ so that $\alpha \beta < 1$, then $n$ dominates $\Delta ^{\beta} = [n^{\alpha}]^{\beta}$ and hence there exists $K_{\ref{part2}}$ so that for large enough $n$, $e^{-n \epsilon}e^{-s\Delta^{\beta}}/4 \geq e^{-K_{\ref{part2}}n \varepsilon}$.
\qed

\section*{Acknowledgment}
Part of this research was
carried out during a visit of the fourth named author to the
Mathematics Department at the University of Southern California, Los
Angeles, in 2012. She thanks this department for its hospitality.


\end{document}